\documentclass{amsart}

\usepackage{amssymb}
\usepackage{amsmath}
\usepackage{amsfonts}
\usepackage{amsthm}
\usepackage{graphicx}

\DeclareMathSymbol{\Z}{\mathalpha}{AMSb}{"5A}
\DeclareMathSymbol{\R}{\mathalpha}{AMSb}{"52}
\DeclareMathSymbol{\C}{\mathalpha}{AMSb}{"43}
\DeclareMathSymbol{\D}{\mathalpha}{AMSb}{"44}
\DeclareMathOperator{\re}{Re} \DeclareMathOperator{\im}{Im}

\newcommand{\KHO}{K_{\scriptscriptstyle H}^{\scriptscriptstyle O}}
\newcommand{\CHO}{C_{\scriptscriptstyle H}^{\scriptscriptstyle O}}
\newcommand{\SHO}{S_{\scriptscriptstyle H}^{\scriptscriptstyle O}}
\newcommand{\cnj}[1]{~\overline{#1}}
\newcommand{\series}[3]{\displaystyle \sum_{{#1}={#2}}^{#3} }

\newtheorem{thm}{Theorem}
\newtheorem{lem}[thm]{Lemma}
\newtheorem*{kthma}{Theorem A}
\newtheorem*{kthmb}{Theorem B}

\theoremstyle{remark}
\newtheorem{rem}{Remark}
\newtheorem{exmp}{Example}
\newtheorem*{cohn}{Cohn's Rule}
\newtheorem*{schurcohn}{Corollary to the Schur-Cohn Algorithm}

\begin{document}

\title[]{Convolutions of harmonic convex mappings}
\date{February 2, 2009}

\author[Dorff]{Michael Dorff}
 \email{mdorff@math.byu.edu}
\address{Department of Mathematics, Brigham Young University,
Provo, Utah, 84602, U.S.A.}

\author[Nowak]{Maria Nowak}
\email{nowakm@golem.umcs.lublin.pl}
\address{Department of Mathematics,
Maria Curie-Sklodowska University, 20-031 Lublin, Poland}

\author[ Wo\l oszkiewicz]{Magdalena Wo\l oszkiewicz}
\email{mwoloszkiewicz@gmail.com}
\address{Department of Mathematics,
Maria Curie-Sklodowska University, 20-031 Lublin, Poland}


\thanks{}
\dedicatory{}

\begin{abstract}
The first author proved that the harmonic convolution of a normalized right half-plane
mapping with either another normalized right half-plane mapping or
a normalized vertical strip mapping is convex in the direction of
the real axis. provided that it is locally univalent. In this
paper, we prove that in general the assumption of local univalency cannot be omitted. However, we are able to show that in some cases these harmonic convolutions are locally univalent. Using this we obtain interesting examples of univalent harmonic maps one of which is a map onto the plane with two parallel slits.
\end{abstract}

\maketitle

\section{Introduction}

Let $\D$ be the unit disk. We consider the family of
complex-valued harmonic functions $f=u+iv$ defined in $\D$, where
$u$ and $v$ are real harmonic in $\D$. Such functions can be
expressed as $f=h+\cnj{g}$, where
\begin{equation*}
h(z)=z+\series{n}{2}{\infty}a_nz^n \hspace{.3in} \mbox{and}
\hspace{.3in} g(z)=\series{n}{1}{\infty}b_nz^n
\end{equation*}
are analytic in $\D$. The harmonic function $f=h+\cnj{g}$ is locally
one-to-one and sense- preserving in $\D$ if and only if
\begin{equation*}
\label{eq:sp} |g'(z)|<|h'(z)|, \forall z \in \D.
\end{equation*}
In such a case, we say that $f$ is locally univalent and $f$ satisfies the dilatation condition. Let $\SHO$ be the
class of complex-valued, harmonic, sense-preserving, univalent
functions $f$ in $\D$, normalized so that $f(0)=0$, $f_{z}(0)=1$, and  $f_{\cnj{z}}(0)=0$.
Let $\KHO$ and $\CHO$ be the subclasses of $\SHO$ mapping $\D$ onto convex and close-to-convex domains, respectively. We will deal with $\CHO$ mappings that are convex in one direction.

For analytic functions $f(z)=z+\sum_{{n}={2}}^{\infty}a_nz^n$ and
$F(z)=z+\sum_{{n}={2}}^{\infty}A_nz^n$, their convolution (or
Hadamard product) is defined as
$f*F=z+\sum_{{n}={2}}^{\infty}a_nA_nz^n$. In the harmonic case,
with
\begin{align*}
f&=h+\cnj{g}=z+\series{n}{2}{\infty}a_nz^n+\series{n}{1}{\infty}
\cnj{b}_n\cnj{z}^n~~~\mbox{and} \\
F&=H+\cnj{G}=z+\series{n}{2}{\infty}A_nz^n+\series{n}{1}
{\infty}\cnj{B}_n\cnj{z}^n,
\end{align*}
define the harmonic convolution as
\begin{equation*}
f*F=h*H+\cnj{g*G}=z+\series{n}{2}{\infty}a_nA_nz^n+\series{n}{1}
{\infty}\cnj{b_nB_n}\cnj{z}^n,
\end{equation*}

There have been some results about harmonic convolutions of
functions (see \cite{css}, \cite{dor}, \cite{goo}, and \cite{rs}). For the convolution of analytic
functions, if $f_1, f_2 \in K$, then $f_1 * f_2 \in K$. Also, the
right half-plane mapping, $\frac{z}{1-z}$, acts as the convolution
identity. In the harmonic case, there are infinitely many right
half-plane mappings and the harmonic convolution of one of these
right half-plane mappings with a function $f \in \KHO$ may not
preserve the properties of $f$, such as convexity or even
univalence (see \cite{dor} for an example).

In \cite{dor}, the following theorems were proved:

\begin{kthma}
\label{thm:rhprhp} Let $f_1=h_1+\cnj{g}_1, f_2=h_2+\cnj{g}_2 \in \KHO$ be
right half-plane mappings. If $f_1*f_2$ satisfies the dilatation
condition, then  $f_1* f_2 \in \SHO$ and is convex in the
direction of the real axis.
\end{kthma}

\begin{kthmb}
\label{thm:rhpavs} Let $f_1=h_1+\cnj{g}_1 \in \KHO$ be a right
half-plane mapping and $f_2=h_2+\cnj{g}_2 \in \KHO$ be a vertical strip
mapping. If $f_1 * f_2$ satisfies the dilatation condition, then
$f_1 * f_2 \in \SHO$ and is convex in the direction of the real
axis.
\end{kthmb}
In section 2, we generalize Theorem A for harmonic mappings onto slanted half-planes given by
\begin{equation*}
H_{\gamma}= \bigg\{ z\in \C: \re (e^{i\gamma}z)>-\frac{1}{2}
\bigg\}, \mbox{ where } 0 \leq \gamma < 2\pi.
\end{equation*}
Next, we deal mainly with the convolution of the canonical harmonic right half-plane mapping (see \cite{css}) given by
\begin{equation}
\label{eq:f0}
f_0(z)=h_0(z)+\overline{g_0(z)}=\dfrac{z-\frac{1}{2}z^2}{(1-z)^2}
-\overline{\dfrac{\frac{1}{2}z^2}{(1-z)^2}}
\end{equation}
with harmonic mappings $f$ that are either right half-planes or strip mappings. We show that if the dilatation of $f$ is $e^{i\theta}z^n$ ($n=1,2$), then $f_0*f$ is locally univalent. However, we give examples when local univalency fails for $n \geq 3$. Also, we provide some results about univalency in the case the dilatation of $f$ is $\frac{z+a}{1+az}$. Finally, we give examples of univalent harmonic maps obtained by way of convolutions.

\section{Convoluting slanted half-plane mappings}

In \cite{as}, \cite{dor1}, and \cite{hs2}, explicit descriptions are
given for half-plane and strip mappings. Specifically, the
collection of functions $f=h+\cnj{g} \in \SHO$ that map $\D$ onto the
right half-plane, $R=\{w: \re (w) > -1/2\}$, have the form
\begin{equation*}
h(z)+g(z)=\dfrac{z}{1-z}
\end{equation*}
and those that map $\D$ onto the vertical strip, $\Omega
_{\alpha}= \Big\{ w: \frac{\alpha -\pi}{2 \sin \alpha} < \re (w) <
\frac{\alpha} {2\sin \alpha} \Big\}$, where $\frac{\pi}{2} \leq
\alpha < \pi$, have the form
\begin{equation*}
h(z)+g(z)=\dfrac{1}{2i\sin \alpha} \log \bigg(
\dfrac{1+ze^{i\alpha}} {1+ze^{-i\alpha}} \bigg).
\end{equation*}

\vspace{.2in}

\noindent Now, we first prove a generalization of Theorem A for the slanted half-plane,
$H_{\gamma}$, $0\leq \gamma<2\pi$,  described in the introduction.
Let $S^0(H_{\gamma}) \subset \SHO$ denote  the class of
harmonic functions  $f$ that map $\D$ onto $H_{\gamma}$. In the
case when $\gamma=0$ we get the normalized class of harmonic
functions that map $\D$ onto the right half-plane $\{w:\re (w)
>-1/2\}$.

\begin{lem} If $f=h+\bar g\in S^0(H_{\gamma})$, then
\begin{equation*}
 h(z)+e^{- 2i \gamma}g(z)=\frac z{1-ze^{i\gamma}},
\quad z\in \D.
\end{equation*}
\end{lem}
\begin{proof}
If $f=h+\bar g\in S^0(H_{\gamma})$, then $\re \big\{
e^{i\gamma}(h(z)+  \overline{g(z)} )  \big\}> -1/2$, which means
that $\re \big\{ e^{i\gamma}h(z)+ e^{-i\gamma}{g(z)} \big\} >
-1/2$. In other words, $\re \big\{
e^{i\gamma}(h(z)+e^{-2i\gamma}{g(z)} )\big\} > -1/2$. Since $f$ is
a convex harmonic function, it follows from a convexity theorem by
Clunie and Sheil-Small \cite{css} that the function
$h(z)+e^{-2i\gamma}{g(z)}$ is  convex in the direction
$\pi/2-\gamma$, and so is univalent. It is also clear, that
$z\mapsto f(z)+e^{-2i\gamma}{g(z)}$ maps $\D$ onto $H_{\gamma}$
which implies result.
\end{proof}

\begin{thm}
\label{thm:genA} If $f_k\in S^0(H_{\gamma_k}) $,  $k=1,2$, and
$f_1*f_2$ is locally univalent in $\D$, then $f_1*f_2$ is convex
in the direction $-(\gamma_1+\gamma_2).$
\end{thm}

\begin{proof}
Let
\begin{align*}
F_1 & = (h_1+e^{-2i\gamma_1}g_1)*(h_2-e^{-2i\gamma_2}g_2), \mbox{ and} \\
F_2 & = (h_2+e^{-2i\gamma_2}g_2)*(h_1-e^{-2i\gamma_1}g_1).
\end{align*}
Then
\begin{equation*}
\frac12(F_1+F_2)=h_1*h_2-e^{-2i(\gamma_1+\gamma_2)}{g_1*g_2}.
\end{equation*}
The shearing theorem of Clunie and Sheil-Small \cite{css}
establishes that it is sufficient to show that the function
$\frac12(F_1+F_2)$ is convex in the direction $
-(\gamma_1+\gamma_2)$, or equivalently, that
$F=e^{i(\gamma_1+\gamma_2)}(F_1+F_2)$ is convex in the direction
of the real axis. A result by Royster and Ziegler \cite{rz} shows
that $F$ is convex in the real direction, if $\re \big \{
(zF'(z))/\varphi (z) \big \}
>0, \forall z \in \D$, where $\varphi (z) = \frac{ze^{i\alpha}}{(1-ze^{i\alpha}z)^2}$ with some
$\alpha\in\mathbb R$. Thus, if we show
this last condition, we are done.

By Lemma 1,
\begin{equation*}
zF_1'(z)= \frac z{1-ze^{i\gamma_1}}*[
z(h_2-e^{-2i\gamma_2}g_2)'(z)].
\end{equation*}
Furthermore,
\begin{align*}
z(h_2-e^{-2i\gamma_2}g_2)'(z)& = z \;
\frac{h_2'(z)-e^{-2i\gamma_2}g_2'(z)}
{h_2'(z)+e^{-2i\gamma_2}g_2'(z)} \; \big(
h_2'(z)+e^{-2i\gamma_2}g_2'(z)\big )
\\ & = \frac{zp_2(z)}{(1-e^{i\gamma_2}z)^2},
\end{align*}
where $\re \{ p_2(z) \} >0$ for all $z\in \D$. Consequently,
\begin{align*}
zF_1'(z) & = \frac
z{1-ze^{i\gamma_1}}*\frac{zp_2(z)}{(1-e^{i\gamma_2}z)^2} \\
& =  e^{-i\gamma_1}\frac
{ze^{i\gamma_1}}{1-ze^{i\gamma_1}}*\frac{zp_2(z)}
{(1-e^{i\gamma_2}z)^2} \\
& = \frac{ zp_2(ze^{i\gamma_1})}{(1-e^{i(\gamma_1+\gamma_2)}z)^2}
\end{align*}
Analogously,
\begin{equation*}
zF_2'(z)=\frac{
zp_1(ze^{i\gamma_2})}{(1-e^{i(\gamma_1+\gamma_2)}z)^2},
\end{equation*}
where $\re \{ p_1(z) \}>0$ for all $z\in \D$. Thus
\begin{equation*}
\text{Re} \left(\frac{e^{i(\gamma_1+\gamma_2)}(zF_1'(z)+zF_2'(z))}
{\frac{
ze^{i(\gamma_1+\gamma_2)}}{(1-e^{i(\gamma_1+\gamma_2)}z)^2}}
\right)= \re (p_1(ze^{i\gamma_2})+p_2(ze^{i\gamma_1}))>0.
\end{equation*}
\end{proof}

\section{Convoluting $f_0$ with right half-plane mappings}

In Theorem A, Theorem B, and Theorem 2, we require that the resulting convolution
function satisfy the dilatation condition
\begin{equation*}
|\omega(z)|=\bigg| \dfrac{g'(z)}{h'(z)} \bigg| < 1, \forall z \in
\D.
\end{equation*}
When is this not a necessary assumption? In the rest of the paper
we establish cases of these theorems for which this assumption can be omitted.

The following result about the number of zeros of polynomials in the disk is helpful in proving the next several theorems.

\begin{cohn}
\label{m:1}
(\cite{cohn} or see \cite{rasc}, p 375)
Given a polynomial
$$f(z)=a_0+a_1z+\cdots+a_nz^n$$ of degree $n$, let
$$f^*(z)=z^n\cnj{f(1/\cnj{z})}=\cnj{a_n}+\cnj{a_{n-1}}z+\cdots+\cnj{a_0}z^n.$$
Denote by $p$ and $s$ the number of zeros of $f$ inside the unit
circle and on it, respectively. If $|a_0|<|a_n|$, then
$$f_1(z)=\dfrac{\cnj{a_n}f(z)-a_0f^*(z)}{z}$$ is of degree $n-1$ with $p_1=p-1$
and $s_1=s$ the number of zeros of $f_1$ inside the unit circle
and on it, respectively.
\end{cohn}

The main result of this section is the following.

\begin{thm}
\label{thm:f0rhp} Let $f=h+\overline{g} \in \KHO$ with
$h(z)+g(z)=\frac{z}{1-z}$ and $\omega(z)=e^{i\theta}z^n$ $(n \in
\Z^+$ and $\theta \in \R)$. If $n=1,2$, then $f_0*f \in \SHO$ and
is convex in the direction of the real axis.
\end{thm}

\begin{proof}
Let the the dilatation of $f_0*f$ be given by
$\widetilde{\omega}=(g_0*g)'/(h_0*h)'$. By Theorem A and by Lewy's
Theorem, we just need to show that $|\widetilde{\omega}(z)|<1,
\forall z \in \D$.

First, note that if $F$ is analytic in $\mathbb D$ and $F(0)=0$,
then from eq. (\ref{eq:f0})
\begin{align}
\label{eq:conveq}
h_0(z)*F(z)  =\frac{1}{2} \big[ F(z) +z F'(z) \big] \\
g_0(z)*F(z)  =\frac{1}{2} \big[ F(z) -z F'(z) \big] . \notag
\end{align}

Also, since $g'(z)=\omega(z) h'(z)$, we know $g''(z)=\omega(z)
h''(z)+\omega'(z) h'(z)$.

Hence
\begin{equation}
\label{eq:genomega}
\widetilde{\omega}(z)=
-\dfrac{zg''(z)}{2h'(z)+zh''(z)} =
\dfrac{-z\omega'(z)h'(z)-z\omega(z)h''(z)}{2h'(z)+zh''(z)}.
\end{equation}

Using $h(z)+g(z)=\frac{z}{1-z}$ and $g'(z)=\omega(z) h'(z)$, we
can solve for $h'(z)$ and $h''(z)$ in terms of $z$ and
$\omega(z)$:
\begin{align*}
h'(z) = & \dfrac{1}{(1+\omega(z))(1-z)^2} \\ \\h''(z) = &
\dfrac{2(1+\omega(z))-\omega'(z)(1-z)}{(1+\omega(z))^2(1-z)^3}. \\
\end{align*}

Substituting these formulas for $h'$ and $h''$ into the equation
for $\widetilde{\omega}$, we derive:
\begin{align}
\label{eq:omega}
\begin{split}
\widetilde{\omega}(z)= &
\dfrac{-z\omega'(z)h'(z)-z\omega(z)h''(z)}{2h'(z)+zh''(z)} \\
    = & -z \dfrac{\omega^2(z)+[\omega(z)-\frac{1}{2}\omega'(z)z]+\frac{1}{2}\omega'(z)}
    {1+[\omega(z)-\frac{1}{2}\omega'(z)z]+\frac{1}{2}\omega'(z)z^2}.
    \\
\end{split}
\end{align}

Now, consider the case in which $\omega(z)=e^{i\theta}z$. Then eq.
(\ref{eq:omega}) yields
\begin{align*}
\hspace{.5in} \widetilde{\omega}(z)= & -z e^{2i\theta} \frac{
\big( z^2+\frac{1}{2}e^{-i\theta}z+\frac{1}{2}e^{-i\theta} \big)}
{\big( 1+\frac{1}{2}e^{i\theta}z+\frac{1}{2}e^{i\theta}z^2 \big)}.
\end{align*}
If we denote
$f(z)=z^2+\frac{1}{2}e^{-i\theta}z+\frac{1}{2}e^{-i\theta}$, then
$f^*(z)=z^2\cnj{f(1/\cnj{z})}=1+\frac{1}{2}e^{i\theta}z+\frac{1}{2}e^{i\theta}z^2$.
In such a situation, if $z_0$ is a zero of $f$, then
$\frac{1}{\cnj{z_0}}$ is a zero of $f^*$. Hence,
\begin{align*}
\hspace{.5in} \widetilde{\omega}(z)= & -ze^{2i\theta}
\dfrac{(z+A)(z+B)}{(1+\cnj{A}z)(1+\cnj{B}z)}.
\end{align*}
By Cohn's Rule we have
\begin{equation*}
f_1(z)= \dfrac{\cnj{a_2}f(z)-a_0f^*(z)}{z} = \frac{3}{4}z+\bigg(
\frac{1}{2}e^{-i\theta}-\frac{1}{4} \bigg).
\end{equation*}
$f_1$ has one zero at $z_0=\frac{1}{3}-\frac{2}{3}e^{-i\theta}
\in \overline{\D}$. So, $f$ has two zeros, namely $A$ and $B$,
with $|A|, |B| \leq1$.

Next, consider the case in which $\omega(z)=e^{i\theta}z^2$. In
this case,
\begin{equation*}
|\widetilde{\omega}(z)|= \big|z^2 \big|  \bigg|
\frac{z^3+e^{-i\theta}}{1+e^{i\theta}z^3}\bigg|  =  |z|^2 < 1.
\end{equation*}
\end{proof}

\begin{rem}
If we assume the hypotheses of the previous theorem with the
exception of making $n \geq 3$, then for each $n$ we can find a
specific $\omega (z)=e^{i\theta}z^n$ such that $f_0*f \notin
\SHO$. For example, if $n$ is odd, let $\omega (z)=-z^n$ and then
eq. (\ref{eq:omega}) yields
\begin{equation*}
\widetilde{\omega}(z)=-z^n \frac{z^{n+1}+\big( \frac{n}{2}-1
\big)z - \frac{n}{2}}{1+\big( \frac{n}{2}-1
\big)z^n-\frac{n}{2}z^{n+1}}.
\end{equation*}
It suffices to show that for some point $z_0 \in \D$,
$|\widetilde{\omega}(z_0)| >1$. Let $z_0=-\frac{n}{n+1} \in \D$.
Then
\begin{align}
\label{eq:notuniv}
\begin{split}
\widetilde{\omega}(z_0) & = \bigg( \frac{n}{n+1} \bigg)^n
\dfrac{\big( \frac{n}{n+1} \big)^{n+1}-\big( \frac{n}{2}-1 \big)
\big( \frac{n}{n+1} \big) -\frac{n}{2}}{1-\big( \frac{n}{2}-1
\big) \big( \frac{n}{n+1} \big)^n -\big( \frac{n}{2} \big) \big(
\frac{n}{n+1} \big)^{n+1}} \\ & \\
& = 1 + \dfrac{ \Big[ \big( \frac{n+1}{n} \big)^n - \big(
\frac{n}{n+1} \big)^{n+1} \Big] + \Big[ 1- \frac{n}{n+1} \Big]}{
\big( \frac{n}{2}-1 \big) + \big( \frac{n}{2} \big) \big(
\frac{n}{n+1} \big) -  \big( \frac{n+1}{n} \big)^n}.
\end{split}
\end{align}
Note that $\big[ \big( \frac{n+1}{n} \big)^n - \big( \frac{n}{n+1}
\big)^{n+1} \big] + \big[ 1- \frac{n}{n+1} \big] > 0$. Also,
$\big( \frac{n}{2}-1 \big) + \big( \frac{n}{2} \big) \big(
\frac{n}{n+1} \big) -  \big( \frac{n+1}{n} \big)^n >0$ since
$\big( \frac{n}{2}-1 \big) + \big( \frac{n}{2} \big) \big(
\frac{n}{n+1} \big) >n-\frac{3}{2}
> e$ and $\big( \frac{n+1}{n} \big)^n$ is an increasing series
converging to $e$. Thus, if $n \geq 5$ is odd,
$|\widetilde{\omega}(z_0)|
>1$. If $n=3$, it is easy to compute that $|\widetilde{\omega}(z_0)|
= \big( \frac{3}{4} \big)^3 \Big| \frac{3^4-\frac{1}{2}\cdot 3
\cdot 4^3 -\frac{3}{2} \cdot 4^4}{4^4-\frac{1}{2} \cdot 3^3 \cdot
4 -\frac{3}{2} \cdot 3^4} \Big|
> 2$. Now, if $n$ is even, let $\omega (z)=z^n$ and
$z_0=-\frac{n}{n+1}$. This simplifies to the same
$\widetilde{\omega}(z_0)$ given eq. (\ref{eq:notuniv}) and the
argument above also holds for $n \geq 6$. If $n=4$,
$|\widetilde{\omega}(z_0)| = \big( \frac{4}{5} \big)^4 \Big|
\frac{4^5-4 \cdot 5^4-2 \cdot 5^5}{5^5-4^4 \cdot 5 -2 \cdot 4^5}
\Big|
> 15$.
\end{rem}

\begin{thm}
\label{thm:f0rhp2} Let $f=h+\overline{g} \in \KHO$ with
$h(z)+g(z)=\frac{z}{1-z}$ and $\omega(z)=\frac{z+a}{1+az}$ with $a
\in (-1,1)$. Then $f_0*f \in \SHO$ and is convex in the direction
of the real axis.
\end{thm}

\begin{proof}
Using eq. (\ref{eq:omega}) with $\omega=\frac{z+a}{1+az}$, where $-1<a<1$, we have
\begin{align*}
\hspace{.5in} \widetilde{\omega}(z)= & -z \frac{ \big(
z^2+\frac{1+3a}{2}z+\frac{1+a}{2} \big)}
{\big( 1+\frac{1+3a}{2}z+\frac{1+a}{2}z^2 \big)} \\
= & -z \dfrac{f(z)}{f^*(z)} \\
= & -z \dfrac{(z+A)(z+B)}{(1+\cnj{A}z)(1+\cnj{B}z)}.
\end{align*}

Again using Cohn's Rule,
\begin{equation*}
f_1(z)= \dfrac{\cnj{a_2}f(z)-a_0f^*(z)}{z} =
\frac{(a+3)(1-a)}{4}z+\frac{(1+3a)(1-a)}{4}.
\end{equation*}
So $f_1$ has one zero at $z_0=-\frac{1+3a}{a+3}$ which is in the
unit circle since $-1 < a< 1$. Thus, $|A|$, $|B| <1$.
\end{proof}

Next, we provide some examples.

\begin{exmp}
Let $f_1=h_1+\cnj{g}_1$, where $h_1+g_1=\frac{z}{1-z}$ with $\omega=z$.
Then
\begin{align*}
h_1 & = \dfrac{1}{4} \log \bigg( \dfrac{1+z}{1-z} \bigg)
+\dfrac{1}{2} \dfrac{z}{1-z} \\
g_1 & = -\dfrac{1}{4} \log \bigg( \dfrac{1+z}{1-z} \bigg)
+\dfrac{1}{2} \dfrac{z}{1-z}.
\end{align*}
Consider $F_1=f_0*f_1=H_1+\cnj{G_1}$. Using eq. (\ref{eq:conveq}) we have
\begin{align*}
H_1 & =h_0*h_1 =\frac{1}{2} [h_1(z)+zh_1'(z)]
     = \dfrac{1}{8} \log \bigg( \dfrac{1+z}{1-z} \bigg)
    +\dfrac{\frac{3}{4}z-\frac{1}{4}z^3}{(1-z)^2(1+z)} \\
G_1 & =g_0*g_1=\frac{1}{2} [g_1(z)-zg_1'(z)] = -\dfrac{1}{8} \log \bigg(
\dfrac{1+z}{1-z} \bigg)
    +\dfrac{\frac{1}{4}z-\frac{1}{2}z^2-\frac{1}{4}z^3}{(1-z)^2(1+z)},
\end{align*}
and from eq. (\ref{eq:omega})
\begin{equation*}
\widetilde{\omega}(z)=-z \bigg( \dfrac{2z^2+z+1}{z^2+z+2} \bigg).
\end{equation*}
 We show that $F_1$ maps the unit disk onto the
domain whose boundary consists of the four  half-lines given by $\{ x\pm
\frac {\pi}8 i,\ x\leq -\frac 14\}$ and $\{-\frac 14+i y,\ |y| \geq \frac
{\pi}8\}$ (see Figure 2). In doing so, we use a similar argument to that used by Clunie and
Sheil-Small in Example 5.4 \cite{css}. We have
$$
F_1(z)=\re\left(  \frac{z-\frac12 z^2-\frac 12 z^3}{(1+z)(1-z)^2}
\right)+ i \im\left(\frac 14\ln\left(\frac{1+z}{1-z}\right)+ \frac
12\frac {z}{(1-z)^2}\right).
$$
Applying the transformation $\zeta=\frac{1+z}{1-z}=\xi+i\eta$,
$\xi>0$, we get
\begin{align*}F_1(z)&=\re
\frac1{8}\left(3\zeta-2-\frac1{\zeta}\right)+i\im \frac14\left(\ln
\zeta+\frac12(\zeta^2-1)\right)\\&= \frac1{8}\left(
3\xi-2-\frac{\xi}{\xi^2+\eta^2}\right)+\frac{i}{4}\left(\arctan
\frac{\eta}{\xi}+\xi\eta\right).\end{align*} Observe first that
the positive real axis $ \{\zeta=\xi+i\eta : \eta=0, \xi>0\}$ is
mapped monotonically onto the whole real axis. Next we find the
images of the level curves
$$\arctan
\frac{\eta}{\xi}+\xi\eta=c,\quad \xi,\eta>0.$$ The polar
coordinates equations of these level curves are
\begin{equation}\label{eq:0vc} \theta +r^2\sin \theta \cos\theta=c,\quad 0<\theta<\frac{\pi}{2}.\end{equation}
Hence \begin{align*}\xi&= \sqrt{(c-\theta)\cot \theta}\\
\eta&= \sqrt{(c-\theta)\tan \theta}, \quad 0<\theta<\min\Big\{ c,
\frac{\pi}2 \Big\}. \end{align*} Fix $c>0$. Then the image of
the curve given in (\ref{eq:0vc}) under $F_1$ is
\begin{align*}
F_1(z)&=\frac{1}{8}\left(3 \sqrt{(c-\theta)\cot \theta} -2
-\sqrt{\frac{\sin\theta\cos^3\theta}{c-\theta}} \right) + \frac
i4\, c\\&= u(c,\theta)+\frac i4 c\,.
\end{align*}
 If $0<c<\frac{\pi}{2}$, then $\theta\in (0,c),$ and one finds easily
 that $\lim_{\theta\to 0^+}u(c,\theta)=\infty$ and $\lim_{\theta\to c-}u(c,\theta)=-\infty$.
 The intermediate value property implies that in this case the image of the  level curve under $F_1$ is the
 entire horizontal line
$\{x+\frac{ic}4: -\infty<x<\infty\}.$
  If $c\geq\frac{\pi}2$, then $\lim_{\theta\to 0^+}u(c,\theta)=\infty$ and
   $\lim_{\theta\to \pi/2^-}u(c,\theta)=-\frac 14$. So in this case the images of   the level
  curves are horizontal half-lines $\{x+\frac{ic}4,: -\frac
  14<x<\infty\}$. This means that images of the level curves under $F_1$
  fill the domain whose boundary consists of the real axis and two
  half-lines
  $\{x +
\frac {\pi}8 i,\ x\leq -\frac 14\}$ and $\{-\frac 14+i y,\ y \geq \frac
{\pi}8\}$. Finally,  our claim follows from the fact that the
range of $F_1$ is symmetric with respect to the real axis.

The images of concentric circles inside $\D$ under the
harmonic maps $f_0$ and under $f_1$ are shown in Figure
\ref{fig:f0andf1}. The images of these concentric circles under
the convolution map $f_0*f_1=F_1$ are shown in Figure 2.

\begin{center}
\begin{figure}[h]
\label{fig:f0andf1}
\begin{tabular}{cc}
\hspace*{-.2in}
\includegraphics[scale=0.3]{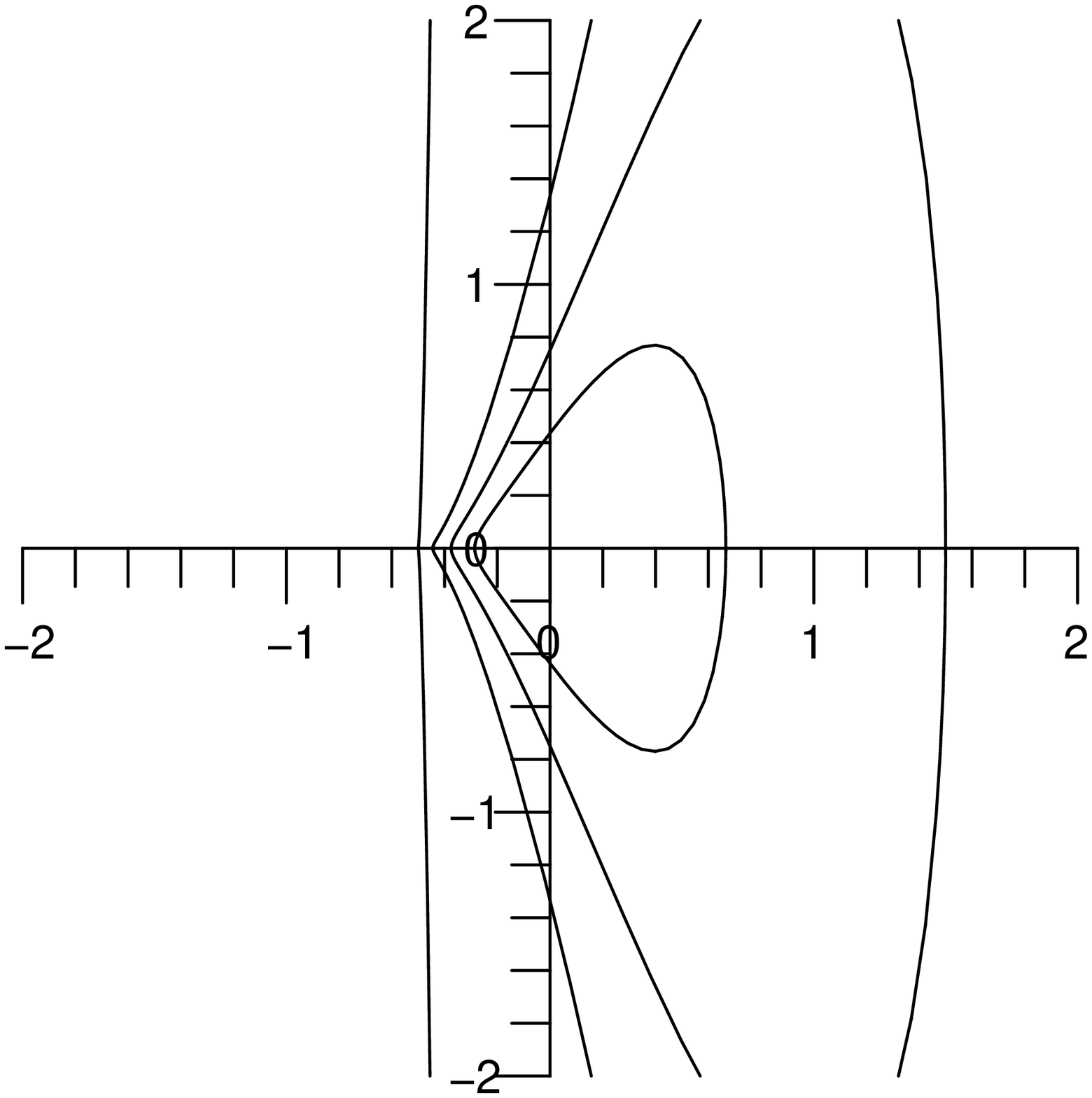} &
\includegraphics[scale=0.3]{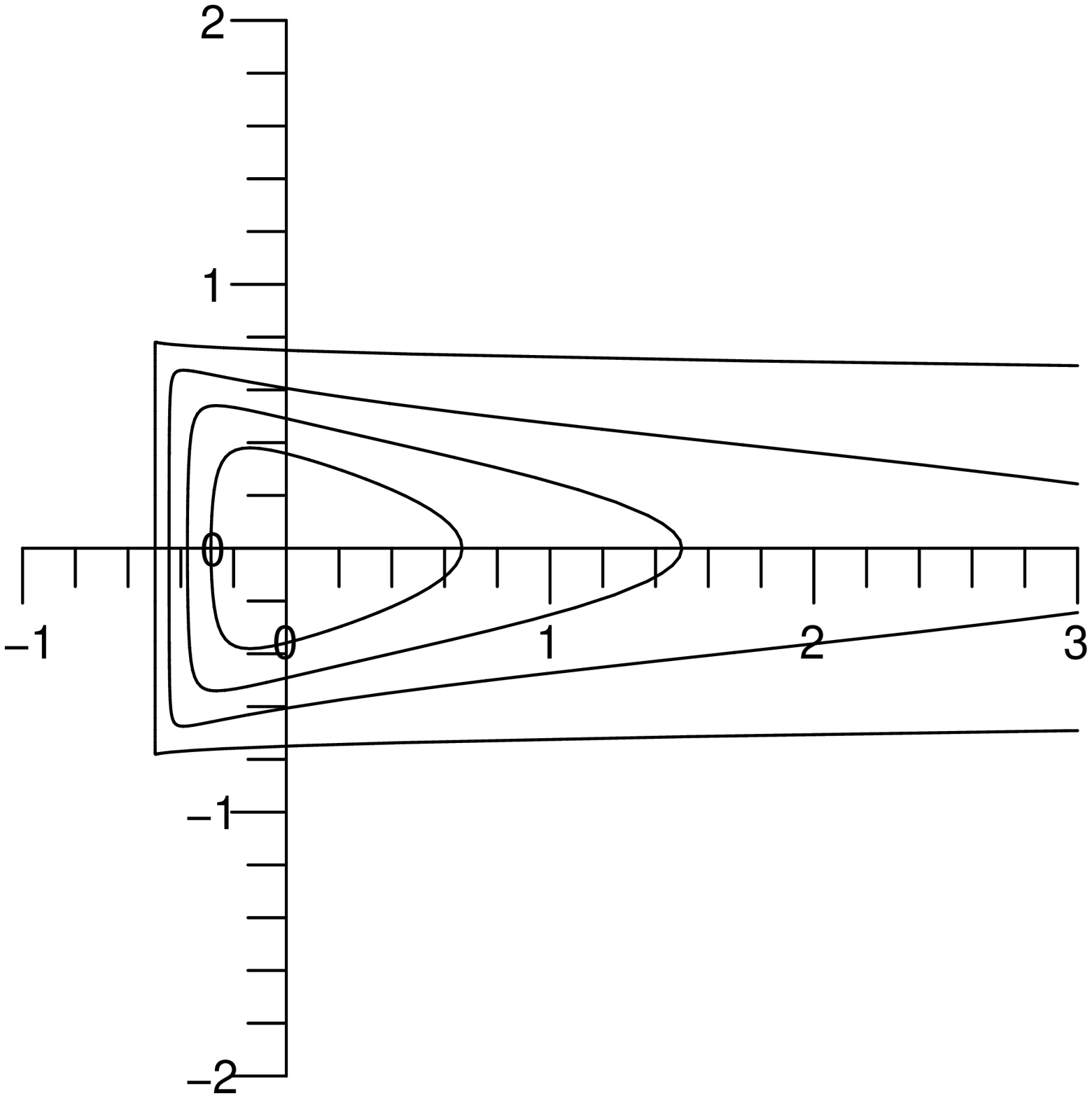} \\
\end{tabular}
\caption{Image of concentric circles inside $\D$ under the maps
$f_0$ and $f_1$, respectively.}
\end{figure}
\end{center}

\begin{center} \begin{figure}[h]
\label{fig:convf1} \hspace{-45pt}
\begin{tabular}{l}
\centerline{\includegraphics[scale=0.3]{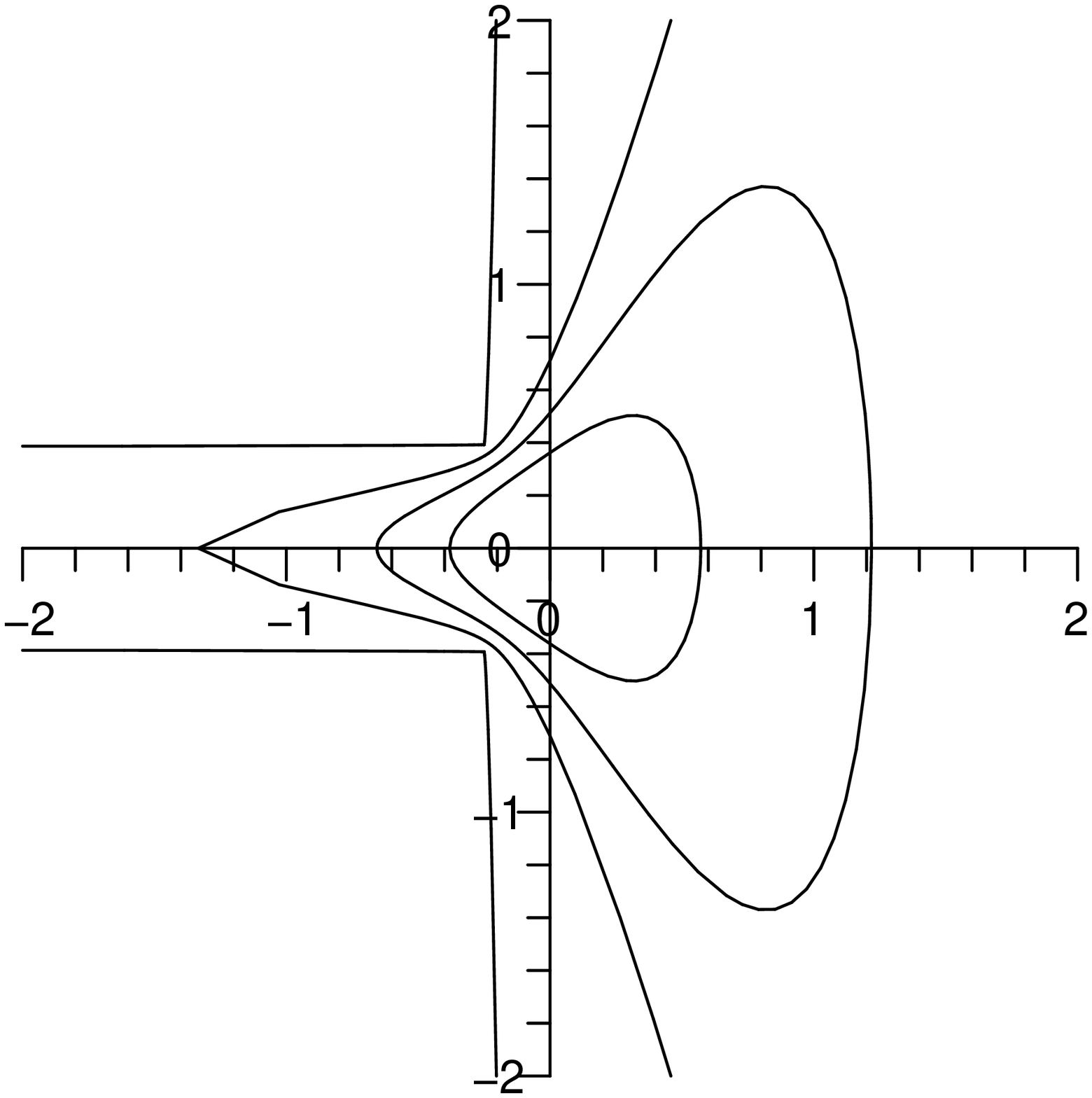}}
\end{tabular}
\caption{Image of concentric circles inside $\D$ under the
convolution map $f_0*f_1=F_1$.}
\end{figure} \end{center}
\end{exmp}

\begin{exmp}
Let $f_2=h_2+\overline{g_2}$ be the harmonic mapping in the disk
$\D$ such that $h_2(z)+g_2(z)= \frac{z}{1-z}$ and $\omega_2
(z)=\frac{g_2'(z)}{h_2'(z)}=-z^2$. One can find that
\begin{align*}
h_2(z)&=\frac18 \ln\left(\frac{1+z}{1-z}\right) +\frac12 \frac{z}{1-z}+\frac14 \frac{z}{(1-z)^2},\\
g_2(z)&=-\frac18 \ln\left(\frac{1+z}{1-z}\right) +\frac12
\frac{z}{1-z}-\frac14 \frac{z}{(1-z)^2}
\end{align*}
and the image of $\D$ under $f_2$ is the right
half-plane, $R=\left\{\omega:\re(\omega)>-\frac12\right\}$. We
note here that $ f_2(e^{it})=-\frac 12 + i\frac {\pi}{16},$ if $
0<t<\pi$ and $ f_2(e^{it})=-\frac 12 - i\frac {\pi}{16},$ if $
\pi<t<2\pi.$
 Next let
$$
F_2= h_0*h_2+\overline{g_0*g_2}=H_2+\overline{G_2}.$$ By eq.
(\ref{eq:conveq})
\begin{align*}
H_2(z)&=\frac{1}{2} \left[\frac18\ln\left(\frac{1+z}{1-z}\right)
+\frac12
\frac{z}{1-z}+\frac14\frac{z}{(1-z)^2}+\frac{z}{(1-z)^3(1+z)}\right],\\
G_2(z)&=\frac{1}{2} \left[-\frac18\ln\left(\frac{1+z}{1-z}\right)
+\frac12
\frac{z}{1-z}-\frac14\frac{z}{(1-z)^2}+\frac{z^3}{(1-z)^3(1+z)}\right]
\end{align*}
and
$$\widetilde{\omega}_2(z)=\frac{G_2'(z)}{H_2'(z)}=z^2$$
Analysis similar to that in Example 1 can be used to  show that
$F_2$ maps the disk onto the plane minus two half-lines given by $
x\pm\frac{\pi}{16}i$, $x\leq -\frac14$.  We have
$$
F_2(z)=\re\left( \frac12\, \frac{z(2-z+z^3)}{(1-z)^3(1+z)}
\right)+i\im \left( \frac18\ln\left(\frac{1+z}{1-z}\right)
+\frac68\frac{z}{(1-z)^2}\right),
$$
which under the same transformation as in Example 1 becomes
$$F_2(z)=
\frac1{16}\left(
\xi^3-3\xi\eta^2+4\xi-4-\frac{\xi}{\xi^2+\eta^2}\right)+\frac{i}{8}\left(\arctan
\frac{\eta}{\xi}+3\xi\eta\right).$$ Analogously,  we find that the
images of the level curves
$$\theta +\frac32\,r^2\sin 2\theta =c,\quad 0<\theta<\frac{\pi}{2}$$  are
\begin{align*}
F_2(z)&=\frac{1}{16}\left[ \sqrt{\frac 13\,(c-\theta)\cot \theta}
\left( (c-\theta)\left(\frac13\cot\theta-\tan\theta\right)  +4-
\frac{3\sin2\theta}{2(c-\theta)}\right) - 4 \right] + \frac i8\,
c\\&= u(c,\theta)+\frac i8 c\,.
\end{align*}
 If $0<c<\frac{\pi}{2}$ (or $c>\frac {\pi}2$, respectively),  then $\lim_{\theta\to c^-}u(c,\theta)=-\infty$ (or  $\lim_{\theta\to \frac{\pi}{2}}u(c,\theta)=-\infty$, respectively) and $\lim_{\theta\to 0^+}u(c,\theta)=\infty$. This means that the images of the level curves are entire horizontal lines. If  $c=\frac{\pi}{2}$, then
 $\lim_{\theta\to 0^+} u(\frac{\pi}2,\theta)=+\infty $ and $\lim_{\theta\to \pi/2 ^-} u(\frac{\pi}2,\theta)=-\frac14 $.
So,
 $F_2$ maps the first quadrant onto the upper half-plane minus the half-line
 $\{ x+i\frac {\pi}{16}:x\leq -\frac 14\}$, and the result
 follows from the symmetry.
\end{exmp}

\section{Convoluting $f_0$ with vertical strip mappings}

In this section we replace right half-plane maps with vertical strip maps and prove the corresponding analogues for Theorem 3 and Theorem 4.

\begin{thm}
\label{thm:f0avs} Let $f=h+\overline{g} \in \KHO$ with
$h(z)+g(z)=\frac{1}{2i\sin \alpha} \log \Big(
\frac{1+ze^{i\alpha}}{1+ze^{-i\alpha}} \Big)$, where
$\frac{\pi}{2} \leq \alpha < \pi$ and $\omega(z)=e^{i\theta}z^n$.
If $n=1,2$, then $f_0*f \in \SHO$ and is convex in the direction of the real
axis.
\end{thm}
\begin{proof}
By Theorem B we need to establish that $f_0*f=H +\cnj{G}$ is
locally univalent. Using $h(z)+g(z)=\frac{1}{2i\sin \alpha} \log
\Big( \frac{1+ze^{i\alpha}}{1+ze^{-i\alpha}} \Big)$ and
$g'(z)=\omega(z)h'(z)$, we get
\begin{align*}
h'(z) = & \dfrac{1}{(1+\omega(z))(1+ze^{i\alpha})(1+ze^{-i\alpha})} \\
\\h''(z) = &
\dfrac{-[2(\cos \alpha+z)(1+\omega(z))+\omega '(z)(1+2\cos \alpha z +z^2)]}
{(1+\omega(z))^2(1+ze^{i\alpha})^2(1+ze^{-i\alpha})^2}. \\
\end{align*}
Substituting these into eq. (\ref{eq:genomega}), yields
\begin{equation}
\label{eq:omega2} \widetilde{\omega}(z)= -z \Bigg\{
\dfrac{\omega^2(z)+\big[\omega(z)-\frac{1}{2}\omega'(z)z
\big]-\frac{1}{2}\omega'(z)\big( \frac{1+\cos \alpha z}{\cos
\alpha +z} \big)}
    {-\big(\frac{1+\cos \alpha z}{\cos \alpha +z}
\big)-\big[\omega(z)-\frac{1}{2}\omega'(z)z \big] \big(
\frac{1+\cos \alpha z}{\cos \alpha +z}
\big)+\frac{1}{2}\omega'(z)z^2} \Bigg\}.
\end{equation}
First, consider the case in which $\omega(z)=e^{i\theta}z$. We have
\begin{align*}
\widetilde{\omega}(z)= & z e^{2i\theta} \dfrac{z^3+(\cos
\alpha+\frac{1}{2}e^{-i\theta})z^2-\frac{1}{2}e^{-i\theta}}{1+(\cos
\alpha+\frac{1}{2}e^{i\theta})z-\frac{1}{2}e^{i\theta}z^3} \\
= & z e^{2i\theta} \dfrac{f(z)}{f^*(z)} \\
= & z e^{2i\theta}
\dfrac{(z+A)(z+B)(z+C)}{(1+\cnj{A}z)(1+\cnj{B}z)(1+\cnj{C}z)}.
\end{align*}

We will show that $A,B,C \in \overline{\D}$.  Let
 \begin{align*}
 \widetilde{\omega}(z)&=ze^{2i\theta}\frac{z^3+(\cos \alpha+\frac12 e^{-i\theta})z^2-\frac12 e^{-i\theta}}{1+(\cos \alpha+\frac12 e^{i\theta})z-\frac12 e^{i\theta}z^3}\\
 &=ze^{2i\theta}\frac{f(z)}{f^*(z)}\\
 &=ze^{2i\theta}\frac{(z+A)(z+B)(z+C)}{(1+\overline{A}z)(1+\overline{B}z)(1+\overline{C}z)},
\end{align*}
where $\alpha\in [\frac\pi 2, \pi)$, $\theta \in[-\pi,\pi]$.
We apply Cohn's Rule to $f(z)=z^3+(\cos \alpha+\frac12 e^{-i\theta})z^2-\frac12 e^{-i\theta}$. Note that $|\frac12 e^{-i\theta}|=\frac12<1$, thus we get
\begin{align*}
f_1(z)=\frac{\overline{a_3}f(z)-a_0f^*(z)}{z}=
\frac34z^2+(\cos \alpha+\frac12 e^{-i\theta})z+\frac12 e^{-i\theta}(\cos \alpha +\frac12 e^{i\theta}).
\end{align*}
Since $\left|\frac12 e^{-i\theta}(\cos \alpha +\frac12
e^{i\theta})\right|\leq \frac12|\cos
\alpha|+\frac14<\frac12+\frac14=\frac34$ (note that $\alpha\neq
\pi$), we can use Cohn's Rule again; this time on $f_1$.

We get
\begin{align*}
f_2(z)&=\frac{\frac34 f_1(z)-\frac12 e^{-i\theta}(\cos \alpha +\frac12 e^{i\theta})f_1^*(z)}{z} \\
&= \left(\frac{9}{16}-\frac14\left|\cos \alpha +\frac12
e^{i\theta}\right|^2\right)z+ \frac34\left(\cos \alpha + \frac12
e^{-i\theta}\right)-\frac12 e^{-i\theta}\left(\cos \alpha +\frac12
e^{i\theta}\right)^2.
\end{align*}
Clearly $f_2$ has one zero at
\begin{align*}
z &= \frac{-\frac34\left(\cos \alpha +\frac12 e^{-i\theta}\right)+\frac12 e^{-i\theta}\left(\cos \alpha +\frac12 e^{i\theta}\right)^2}{\frac{9}{16}-\frac14\left|\cos \alpha +\frac12 e^{i\theta}\right|^2} = \frac{-\frac14\cos\alpha +\frac12 e^{-i\theta}\cos^2 \alpha -\frac38 e^{-i\theta} +\frac18e^{i\theta}}{\frac12-\frac14 \cos^2\alpha -\frac14 \cos \alpha \cos \theta}.
\end{align*}
We show that $|z|\leq 1 $, or equivalently,
$$\left|-\frac14\cos\alpha +\frac12 e^{-i\theta}\cos^2 \alpha
-\frac38 e^{-i\theta} +\frac18e^{i\theta}\right|^2\leq
\left|\frac{1}{2}-\frac14 \cos^2\alpha -\frac14 \cos \alpha \cos
\theta\right|^2.$$ If we put $x=\cos \alpha$,  $y=\cos \theta$,
then $x \in(-1,0], y\in[-1,1]$ and the above inequality  becomes
\begin{align*}
-\frac{3}{16}x^4+\frac{3}{16}x^2+\frac{6}{16}x^3y-\frac{6}{16}xy-\frac{3}{16}x^2y^2+\frac{3}{16}y^2=
\frac3{16}(1-x^2)(x-y)^2\geq0,
\end{align*}

Therefore, by Cohn's Rule, $f$ has all its 3 zeros in $\overline{\D}$,
that is $A,B,C\in \overline{\D}$ and so $ |\widetilde{\omega}(z)|<1$ for all $z\in\D$.

Next, consider the case in which $\omega(z)=e^{i\theta}z^2$. In this case,
\begin{equation*}
\widetilde{\omega}(z)= -z^2 e^{i\theta} \Bigg\{ \dfrac{e^{i\theta}z^3-\Big( \dfrac{1+\cos \alpha z}{\cos \alpha +z} \Big)}{- \Big( \dfrac{1+\cos \alpha z}{\cos \alpha +z} \Big)+e^{i\theta}z^3}    \Bigg\}
=-z^2e^{i\theta}.
\end{equation*}
Hence, $| \widetilde{\omega}(z) | < 1$.
\end{proof}

In proving the last theorem, we will use the following corollary of the Schur-Cohn Algorithm.

\begin{schurcohn}
\label{m:2} (see \cite{rasc}, p
383) Given a polynomial
$$f(z)=a_0+a_1z+\cdots+a_nz^n$$ of degree $n$, let
\begin{equation*}
M_{\nu} = \mbox{det} \begin{pmatrix} B_{\nu}^* & A_{\nu} \\
A_{\nu}^* & B_{\nu} \end{pmatrix} \hspace{.2in} (\nu=1, \ldots,n),
\end{equation*}
where $A^*=\cnj{A}^{\top}$ is the conjugate transpose of $A$, and
$A_{\nu}$ and $B_{\nu}$ are the triangular matrices
\begin{equation*}
\begin{small}
A_{\nu}=\begin{pmatrix} a_0 & a_1 & \cdots & a_{\nu -1} \\ & a_0 &
\cdots & a_{\nu -2} \\ & & \ddots & \vdots \\ & & & a_0
\end{pmatrix}, \hspace{.2in}
B_{\nu}=\begin{pmatrix} \cnj{a}_n & \cnj{a}_{n-1} & \cdots & \cnj{a}_{n-\nu +1} \\
& \cnj{a}_n & \cdots & \cnj{a}_{n-\nu+2} \\ & & \ddots & \vdots \\
& & & \cnj{a}_n
\end{pmatrix}
\end{small}
\end{equation*}
The $f$ has all of its zeros inside the unit circle if and only if
the determinants $M_1, \ldots, M_n$ are all positive.
\end{schurcohn}

\begin{thm}
Let $f=h+\overline{g} \in \KHO$ with
$h(z)+g(z)=\frac{1}{2i\sin \alpha} \log \Big(
\frac{1+ze^{i\alpha}}{1+ze^{-i\alpha}} \Big)$, where
$\frac{\pi}{2} \leq \alpha < \pi$  and $\omega(z)=\frac{z+a}{1+az}$ with $a
\in [0,1)$. Then $f_0*f \in \SHO$ and is convex in the direction
of the real axis.
\end{thm}

\begin{proof}
Using eq. (\ref{eq:omega2}) with $\omega=\frac{z+a}{1+az}$ and simplifying, we have
\begin{align*}
\hspace{.5in} \widetilde{\omega}(z)= & z \frac{ \Big\{
z^3+(\frac{1}{2}+\frac{3}{2}a+\cos \alpha)z^2+(a+2a\cos \alpha)z+(a\cos \alpha +\frac{1}{2}a-\frac{1}{2}) \Big\}}{\Big\{
1+(\frac{1}{2}+\frac{3}{2}a+\cos \alpha)z+(a+2a\cos \alpha)z^2+(a\cos \alpha +\frac{1}{2}a-\frac{1}{2})z^3 \Big\}} \\
= & -z \dfrac{f(z)}{f^*(z)} \\
= & -z \dfrac{(z+A)(z+B)(z+C)}{(1+Az)(1+Bz)(1+Cz)}.
\end{align*}

By the Corollary to the Schur-Cohn Algoritm, we need to show that the determinants
$M_1,M_2,M_3$ are all positive (for convenience, let $\cos \alpha = x$; so $-1 < x \leq 0$ and $-1<a<1$):

\begin{align*}
M_1
= & \mbox{det} \begin{pmatrix} a_3 & a_0 \\ \cnj{a}_0 &
\cnj{a}_3
\end{pmatrix} = \mbox{det} \begin{pmatrix} 1 & ax+\frac{1}{2}a-\frac{1}{2} \\
ax+\frac{1}{2}a-\frac{1}{2} & 1 \end{pmatrix} = \frac{1}{4}(2ax+a+1)(3-2ax-a)>0,
\\ &
\\
M_2= & \mbox{det} \begin{pmatrix} a_3 & 0 & a_0 & a_1 \\
a_2 & a_3 & 0 & a_0 \\ \cnj{a}_0 & 0 & \cnj{a}_3 & \cnj{a}_2 \\
\cnj{a}_1 & \cnj{a}_0 & 0 & \cnj{a}_3 \end{pmatrix} =
\mbox{det} \begin{pmatrix} 1 & 0 & ax+\frac{1}{2}a-\frac{1}{2} & a+2ax \\
\frac{1}{2}+\frac{3}{2}a+x & 1 & 0 & ax+\frac{1}{2}a-\frac{1}{2} \\
ax+\frac{1}{2}a-\frac{1}{2} & 0 & 1 & \frac{1}{2}+\frac{3}{2}a+x \\
a+2ax & ax+\frac{1}{2}a-\frac{1}{2} & 0 & 1
\end{pmatrix} \\ = & \frac{1}{4}(1-x)(1-a)(1-2ax-a)(2+4ax+4a+x-2a^2x^2-5a^2x-2a^2-2ax^2)>0,
\end{align*}

if $P(a,x)=2+4ax+4a+x-2a^2x^2-5a^2x-2a^2-2ax^2>0$. We will show that $P(x,a) > 0$ for $0 \leq a<1$ and $-1<x \leq 0$ (although it seems to be true for $-\frac{1}{3} < a <1$). Now
\begin{equation*}
\frac{\partial }{\partial x}P(a,x) = 4a+1-4a^2x-5a^2-4ax=\big[ 4a(1-x) -4a^2(1+x)\big] +\big[ 1-a^2 \big] >0,\end{equation*}
since $0<a<1$ and $-1<x \leq 0$.
Assume $a=a_0>0$ is fixed.
Then, $P(a_0,x)$ is increasing and attains its minimum at $x=-1$. Thus,
\begin{equation*}
P(a_0,x)>P(a_0,-1)=(a_0-1)^2>0.
\end{equation*}
Note, $P(0,x)=2+x>0$.

\begin{align*}
M_3= & \mbox{det} \begin{pmatrix} a_3 & 0 & 0 & a_0 & a_1 & a_2 \\
a_2 & a_3 & 0 & 0 & a_0 & a_1 \\ a_1 & a_2 & a_3 & 0 & 0 & a_0 \\
\cnj{a}_0 & 0 & 0 & \cnj{a}_3 & \cnj{a}_2 & \cnj{a}_1 \\
\cnj{a}_1 & \cnj{a}_0 & 0 & 0 & \cnj{a}_3 & \cnj{a}_2 \\ \cnj{a}_2
& \cnj{a}_1 & \cnj{a}_0 & 0 & 0 & \cnj{a}_3  \end{pmatrix} \\
 = & \mbox{det} \begin{pmatrix} 1 & 0 & 0 & ax+\frac{1}{2}a-\frac{1}{2} & a+2ax & \frac{1}{2} +\frac{3}{2}a+x \\ \frac{1}{2} +\frac{3}{2}a+x  & 1 & 0 & 0 & ax+\frac{1}{2}a-\frac{1}{2} & a+2ax \\
 \frac{1}{2} +\frac{3}{2}a+x  & a+2ax & 1 & 0 & 0 & ax+\frac{1}{2}a-\frac{1}{2} \\
 ax+\frac{1}{2}a-\frac{1}{2} & 0 & 0 & 1 & \frac{1}{2} +\frac{3}{2}a+x  & a+2ax  \\
 a+2ax  & ax+\frac{1}{2}a-\frac{1}{2} & 0 & 0 & 1 & \frac{1}{2} +\frac{3}{2}a+x  \\
\frac{1}{2} +\frac{3}{2}a+x & a+2ax  & ax+\frac{1}{2}a-\frac{1}{2} & 0 & 0 & 1 \\
 \end{pmatrix} \\
= & \frac{1}{4} (x+1)(1-x)^3(1-a)^3(1-2ax-a)^2(1+3a) > 0.
\end{align*}
Therefore, $A,B,C\in \D$ and $ |\widetilde{\omega}(z)|<1$ for all $z\in\D$.
\end{proof}

\begin{rem}
Unlike Theorem \ref{thm:f0rhp2}, this result does not hold for $-1<a<-\frac{1}{3}$ since $M_3 <0$ for these values of $a$.
\end{rem}

\begin{exmp}
Let $f_3=h_3+\cnj{g}_3$, where $h_3+g_3=\frac{1}{2i} \log \big( \frac{1+iz}{1-iz} \big)$ (that is, $\alpha = \frac{\pi}{2}$ in Theorem \ref{thm:f0avs})  with $\omega=-z^2$.
Then
\begin{align*}
h_3 & = \dfrac{1}{4} \log \bigg( \dfrac{1+z}{1-z} \bigg)
- \dfrac{i}{4} \log \bigg( \dfrac{1+iz}{1-iz} \bigg) \\
g_3 & = -\dfrac{1}{4} \log \bigg( \dfrac{1+z}{1-z} \bigg)
- \dfrac{i}{4} \log \bigg( \dfrac{1+iz}{1-iz} \bigg).
\end{align*}
Consider $F_3=f_0*f_3=H_3+\cnj{G_3}$. From eq. (\ref{eq:conveq}) we derive
\begin{align*}
H_3 & =h_0*h_3 =\frac{1}{2} [h_3(z)+zh_3'(z)]
     = \dfrac{1}{8} \log \bigg( \dfrac{1+z}{1-z} \bigg)
  - \dfrac{i}{8} \log \bigg( \dfrac{1+iz}{1-iz} \bigg) +\frac{1}{2}\frac{z}{1-z^4} \\
G_3 & =g_0*g_3=\frac{1}{2} [g_3(z)-zg_3'(z)] = -\dfrac{1}{8} \log \bigg( \dfrac{1+z}{1-z} \bigg)
  - \dfrac{i}{8} \log \bigg( \dfrac{1+iz}{1-iz} \bigg) +\frac{1}{2}\frac{z^3}{1-z^4}.
\end{align*}
From eq. (\ref{eq:omega2}), $\widetilde{\omega}(z)=z^2$.

We now show that the image of the first quadrant of  $\D$
under the mapping $F_3$ is the domain whose boundary consists of the
positive real axis, upper imaginary axis and the lines $\{\frac
{\pi}8 +iy ,\ y\geq \frac{\pi}8\}$,  $\{x+ \frac {\pi}8 i ,\ x\geq
\frac{\pi}8\}$. We have
$$F_3(z)=\re\left( -\frac{i}{4} \log \left( \frac{1+iz}{1-iz} \right)
+\frac{1}{2}\frac{z}{1-z^2}\right) +i\im\left(\frac{1}{4} \log
\left( \frac{1+z}{1-z} \right)
   +\frac{1}{2}\frac{z}{1+z^2}\right). $$
As in the previous two examples, we use the transformation
$\zeta=\frac{1+z}{1-z}=\xi+i\eta$, $\xi>0$, This transformation
maps the part of the disk in the first quadrant onto the exterior of the unit disk
contained in the first quadrant, and  we note that the interval
[0,i) is mapped onto the quarter of the unit circle. If  we put
$\zeta=r^{i\theta}$, $r\geq 1,\ \theta\in[0,\pi/2),$ then we get
\begin{align*}
\re F_3(z)&= \frac 14\left(\arctan\frac{r-\frac
1r}{2\cos\theta}+\frac 12\left(r-\frac
1r\right)\cos\theta\right)\\
\im F_3(z)&= \frac
14\left(\theta+\frac{2\sin2\theta}{\left(r-\frac
1r\right)^2+4\cos^2\theta}\right).\end{align*} One can see
that the image of the quarter of the unit circle in the first quadrant in the
$\zeta$-plane under $F_3$ is the upper imaginary axis and the
image of the line $\xi>1$ is the positive real axis. Now we
consider  the level curves
$$
\theta+\frac{2\sin2\theta}{\left(r-\frac
1r\right)^2+4\cos^2\theta}=c,\quad  c>0.$$ Since $r>1$ and
$\theta\in (0,  \pi/2)$, get  we get from the above
\begin{equation}
\label{eq:exam3}
r-\frac 1r=2\cos\theta\sqrt{\frac{\tan \theta}{c-\theta}-1}.
\end{equation}
Let $\theta_c\in(0,\pi/2)$ be the number  satisfying the equation
$\tan\theta_c=c-\theta_c$. If $0<c<\pi/2$, we assume that $\theta_c<\theta<c$, while if $c\geq \pi/2$, we assume that $\theta_c<\theta<\pi/2$. Using eq. (\ref{eq:exam3}) we find that on the level
curve we  have
$$\re F_3=\frac 14\left(\arctan\sqrt{\frac{\tan
\theta}{c-\theta}-1}+\cos^2 \theta \sqrt{\frac{\tan
\theta}{c-\theta}-1}\right)$$

Using an analysis similar to the one in the previous examples, we get the result.

The images of concentric circles inside $\D$ under the convolution map
$f_0*f_3=F_3$ are shown in Figure 3.

\begin{center} \begin{figure}[h]
\hspace{-45pt}
\begin{tabular}{l}
\centerline{\includegraphics[scale=0.3]{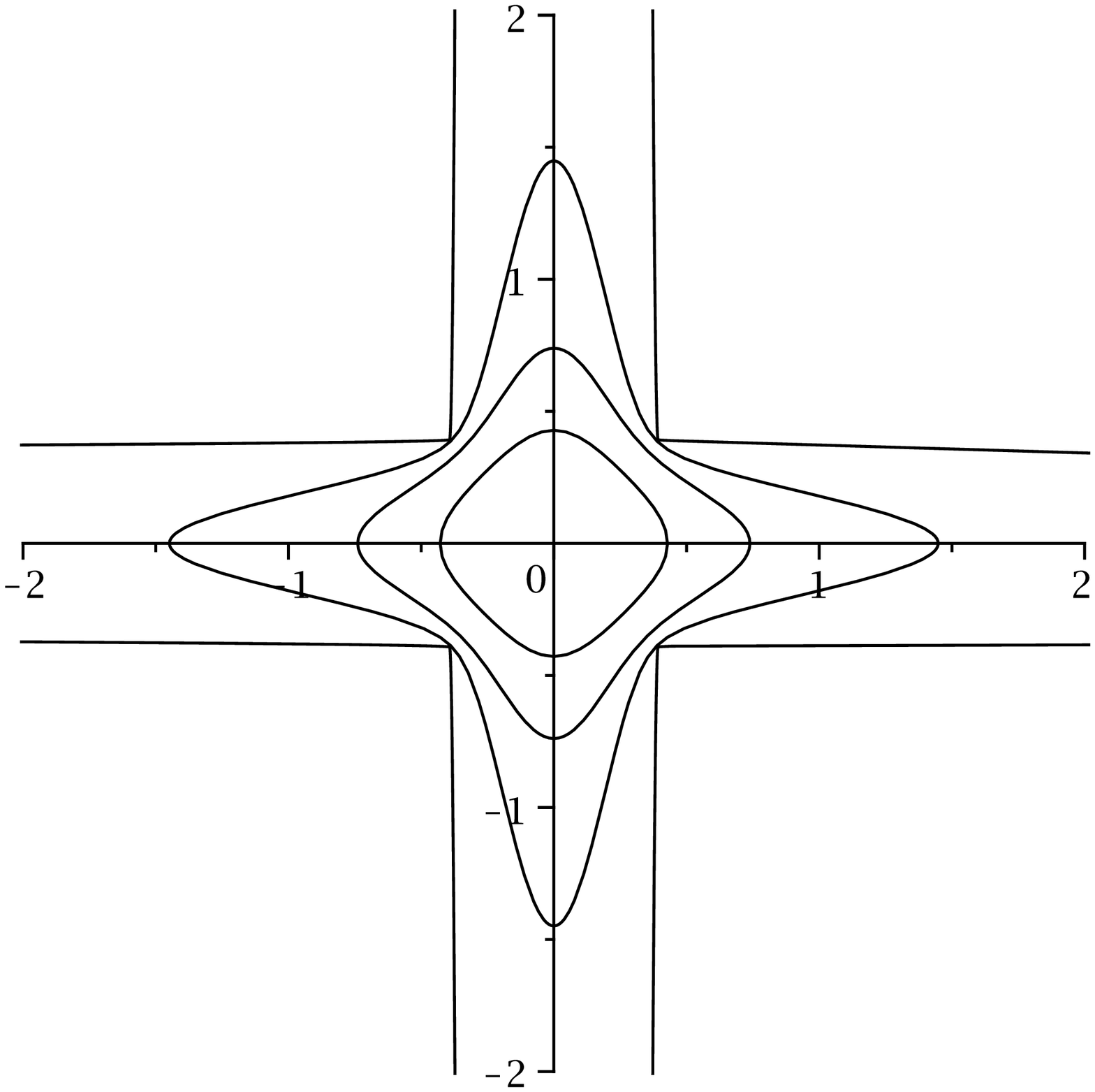}}
\end{tabular}
\caption{Image of concentric circles inside $\D$ under the
convolution map $f_0*f_3=F_3$.}
\end{figure} \end{center}
\end{exmp}

\end{document}